\documentclass[12pt]{article}
\usepackage{amsmath,amsfonts,amssymb,amsthm,amscd}
\usepackage[]{hyperref}
\newtheorem{thm}{Theorem}[section]

\newcounter{alphthm}
\newtheorem{theo}[alphthm]{Theorem}

\def\bd{\begin{definition}}
\def\ed{\end{definition}}
\def\bth{\begin{theorem}}
\def\eth{\end{theorem}}
\def\bpf{\begin{proof}}
\def\epf{\end{proof}}
\def\bco{\begin{corollary}}
\def\eco{\end{corollary}}
\def\ble{\begin{lemma}}
\def\ele{\end{lemma}}
\def\bpr{\begin{proposition}}
\def\epr{\end{proposition}}
\def\bAx{\begin{Axiom}}
\def\eAx{\end{Axiom}}
\newcommand{\be}{\begin{equation}}
\newcommand{\ee}{\end{equation}}
\newcommand{\bes}{\begin{equation*}}
\newcommand{\ees}{\end{equation*}}
\newcommand{\br}{\begin{remark}}
\newcommand{\er}{\end{remark}}
\newcommand{\ben}{\begin{enumerate}}
\newcommand{\een}{\end{enumerate}}
\newcommand{\beq}{\begin{eqnarray}}
\newcommand{\eeq}{\end{eqnarray}}
\newcommand{\beqn}{\begin{eqnarray*}}
\newcommand{\eeqn}{\end{eqnarray*}}
\def\nn{\nonumber}

\def\p{\partial}

\def\p{\partial}

\newtheorem{theorem}{Theorem}[section]
\newtheorem{lemma}[theorem]{Lemma}
\newtheorem{proposition}[theorem]{Proposition}
\newtheorem{corollary}[theorem]{Corollary}
\theoremstyle{definition}
\newtheorem{definition}[theorem]{Definition}

\newtheorem{Axiom}[theorem]{Axiom}
\theoremstyle{remark}
\newtheorem{remark}[theorem]{Remark}

\title{ The Schwarzian derivative on Finsler manifolds of constant curvature}
\author{\small B. BIDABAD\thanks{The corresponding author, bidabad@aut.ac.ir}, \  F. SEDIGHI}
\date{ }
\begin{document}
\maketitle
\begin{abstract}
Lagrange introduced the notion of Schwarzian derivative and Thurston discovered its mysterious properties playing a role similar to that of curvature on Riemannian manifolds.
 Here we continue our studies on the development of the Schwarzian derivative on Finsler manifolds.
First, we obtain an integrability condition for the M\"{o}bius equations. Then we obtain a rigidity result as follows; Let $(M, F)$ be a connected complete Finsler manifold of positive constant Ricci curvature. If it admits non-trivial M\"{o}bius mapping, then $M$ is homeomorphic to the $n$-sphere.
Finally, we reconfirm Thurston's hypothesis for complete Finsler manifolds and show that the Schwarzian derivative of a projective parameter plays the same role as the Ricci curvature on theses manifolds and could characterize a Bonnet-Mayer-type theorem.
\end{abstract}
\vspace{1cm}\leftline{{\textbf{AMS} subject Classification 2010: 53C60, 58B20.}}
\section{Introduction}
  Historically, the definition and elementary properties of the Schwarzian derivative are first discovered by Lagrange in 1781.
The Schwarzian derivative of an injective real function $g$ on $\mathbb{R}$ is defined by
\[ S(g)=\frac{g'''}{g'}-\frac{3}{2}(\frac{g''}{g'})^2,\]
where $ g',g'',g''', $ are the first, the second, and the third derivatives of $g$ with respect to $x\in \mathbb{R}$.
There is a special type of conformal transformations on a Complex plane denoted by $T(z)=\frac{az+b}{cz+d}$, where $ad-bc\neq0,$ called the \textit{M\"{o}bius} transformation. They are characterized by  vanishing of the Schwarzian derivative $S(T)$ of $T$, that is $S(T)=0$.

The Schwarzian derivative is generalized for Riemannian manifolds by Carne \cite{Ca}, Osgood and Stowe \cite{OS}, etc. Thurston discovered that a conformal mapping into the Riemann sphere that has Schwarzian derivative uniformly near zero must be a M\"{o}bius transformation \cite{Th}.

 Recently, we have discussed several applications of Schwarzian in Finsler geometry which confirms Thurston's viewpoint on the mysterious role of this derivative,  for instance, the present authors have studied the Schwarzian of the projective parameter $p$ and classified some Randers manifolds.
    \begin{theo}\cite{B.S}\label{Th;BiSe}
    Let $ (M,F) $ be a compact boundaryless Einstein Randers manifold with constant Ricci scalar and the projective parameter  $p$.\\
\textbullet If $ S(p)=0 $, then $ (M,F) $ is Berwaldian.\\
\textbullet If $ S(p)<0 $, then $ (M,F) $ is Riemannian.
    \end{theo}
    Theorem \ref{Th;BiSe} confirms Thurston's guess in a particular case. In fact, it  shows that the Schwarzian $S(p)$ of a projective parameter $p$, plays a similar role as Ricci curvature on Einstein Randers manifolds, see \cite[Proposition A ]{B.S}.

  As another application a short proof for a known result of Z. Shen in Mathematische Annalen \cite{Sh2} is given in \cite{SBi}, where we proved
 that two projectively related complete Einstein Finsler spaces with constant negative Ricci scalar are homothetic.
We have also obtained the following result using Schwarzian;
\begin{theo}\cite{BiS2}
    Every complete Randers metric of constant negative Ricci curvature (in particular, of constant negative flag curvature) is Riemannian.
     \end{theo}
For more characteristics of Schwarzian derivative in Finsler geometry, see for instance \cite{B.S,BiS,BiS2,SBi}, etc.
 Using some results on \cite{Ru} and \cite{ShYa} we see 
 M\"{o}bius mappings leave invariant Einstein Finsler spaces and Finsler spaces of scalar curvatures.

 In the present work, the Schwarzian derivative on Finsler manifolds is discussed and
  the following results are obtained. Let's set a tensor field $Z$ with the components
\be\label{tensor Z}
      Z^h_{ijk}:=R^h_{ijk}-\frac{1}{n-1}(g_{ij}R^h_{k}-g_{ik}R^h_{j}).
\ee
\bth\label{Th;Z=0}
Let $(M,F)$ be a Finsler manifold. The M\"{o}bius  partial differential equations 
of the conformal factor $\varphi$, is completely integrable, if and only if the tensor $Z$ vanishes.
\eth
  Theorem \ref{Th;Z=0} leads to the following result.
\bth\label{Th;Johari}
    Let $ (M,F) $ be a complete connected Finsler manifold of constant Ricci curvature $c^2$. If $ M $ admits a nontrivial M\"{o}bius mapping, then it is homeomorphic to the n-sphere.
    \eth

   The following theorem reconfirms Thurston's hypothesis and shows that the Schwarzian derivative of a projective parameter plays a same role as the Ricci curvature on the complete Finsler manifolds.
    \bth\label{Th;Bonnet-Myers}
     Let $(M,F)$ be a connected forward-complete  Finsler manifold with constant Ricci scalar. If the Schwarzian $S(p)$ of the projective parameter $p$ satisfies  $S(p)\geqslant 2F^2\lambda$ for $\lambda>0$, then $M$ is  compact and the following holds;\\
       $(i)$ Every geodesic of length at least $ \pi/ \sqrt{\lambda} $ contains conjugate points.\\
       $ (ii) $ The diameter of $ M $ is at most $ \pi/ \sqrt{\lambda} $.\\
       $ (iii) $ The fundamental group $ \pi(M,x) $ is finite.
    \eth

\section{Preliminaries}
\subsection{Notations and elementary definitions}
Let $(M,F)$ be an $n$-dimensional connected smooth Finsler manifold. We denote by  $TM$ the tangent bundle and $\pi:TM_0\to M$, the fiber bundle of non-zero tangent vectors.
Let $(x^i,U)$ be a local coordinate system on $M$ and $(x^i,y^i)$ an element of $TM$.
  Every \emph{Finsler structure} $F$ induces a \emph{spray vector field} $ G:=y^i\frac{\partial}{\partial x^i}-G^i(x,y)\frac{\partial}{\partial y^i},$   on $TM$, where $G^i(x,y)=\frac{1}{4}g^{il}\{[F^2]_{x^ky^l}y^k-[F^2]_{x^l}\}$.
  The vector field $G$ is globally defined on $TM$ in the sense that its components remain invariant after a coordinate change.
  By $ TTM_0 $ we denote the tangent bundle of $TM_0$ and  by $ \pi^*TM $ the pull back bundle of $ \pi $.

    Consider the canonical linear mapping $ \varrho:TTM_0 \to \pi^*TM $, where $ \varrho=\pi_*$  and $ \varrho \hat{X}=X$ for all $ \hat{X}\in \Gamma(TM_0)$.
  Let us denote by $ V_zTM $ the set of \emph{vertical vectors} at $ z=(x,y) \in TM_0 $ equivalently, $ V_zTM=ker\pi_* $  and  $VTM:= \bigcup_{z \in TM_0}V_zTM$ the \emph{bundle of vertical vectors}.
      There is a horizontal distribution $HTM$ such that we have the \textit{Whitney sum} $ TTM_0=HTM\oplus VTM$.        This decomposition permits to write a vector field $\hat{X} \in TTM_0 $ into the horizontal and vertical form $ \hat{X}=H\hat{X}+V\hat{X} $, in a unique manner.

    One can observe that the pair $\{\frac{\delta}{\delta x^i},\frac{\p}{\p y^i}\}$ defined by $\frac{\delta}{\delta x^i}:=\frac{\p}{\p x^i}-G^j_i\frac{\p}{\p y^j}$, where  $ G^j_i:=\frac{\p G^j}{\p y^i}$  forms a horizontal and vertical frame for $TTM$. The  horizontal and vertical dual frame are given by the pair $\{{d x^i},{\delta y^i}\}$.
    One can show that $ 2G^i=\gamma^i_{jk}y^jy^k $, where
  $\gamma^i_{jk}=\frac{1}{2}g^{ih}(\frac{\p g_{hk}}{\p x^j}+\frac{\p g_{hj}}{\p x^k}-\frac{\p g_{jk}}{\p x^h}),$ are formal \textit{Christoffel symbols}.
  The \emph{Cartan connection}'s $1$-form is denoted by $ \omega^{i}_{j}:=\Gamma^{i}_{jk}dx^{k}+C^{i}_{jk}\delta y^{k},$ where
   \be \label{Eq;Chris.Sym.}
 \Gamma^{i}_{jk}:=\frac{1}{2}g^{il}(\delta _{j}g_{lk}+\delta_{k}g_{jl}-\delta _{l}g_{jk}),\quad  C^{i}_{jk}:=\frac{1}{2}g^{il}\dot{\partial}_{l}g_{jk},
   \ee
 and
 $ \delta _{i}:=\frac{\delta}{\delta x^{i}} $,
 $ \dot{\p}_{i}:=\frac{\partial}{\partial y^{i}}.$
In a local coordinate system,
the \textit{horizontal} and \textit{vertical}  Cartan covariant derivatives of an arbitrary $(1,1)$-tensor field on $\pi^*TM$ with the components $T_{i}^{j}$ are denoted here by
\[
 {}^c\nabla_{k}T_{i}^{j}=\delta_{k}T_{i}^{j}-T_{r}^{j}\Gamma ^{r}_{ik}+T_{i}^{r}\Gamma ^{j}_{rk},\qquad
 {}^c\dot{\nabla}_{k}T_{i}^{j}=\dot{\partial }_{k}T_{i}^{j}-T_{r}^{j}C ^{r}_{ik}+T_{i}^{r}C ^{j}_{rk}.\nn
\]
We denote here  by $\overset{\ _*}{R}^i_{jkm}$ the components of \emph{Cartan hh-curvature} tensor.
  They are related to the components of \emph{Chern hh-curvature} tensor $R^i_{jkm}$ by  the following relation \cite{JoBi}.
\begin{equation} \label{R}
\overset{\ _*}{R}^{i}_{jkm}=R^i_{jkm} +R^{s}_{km}C^{i}_{s j},
\end{equation}
where, $R^{i}_{km}:=y^j R^{i}_{ jkm}=\frac{\delta G^i_k}{\delta x^m}-\frac{\delta G^i_m}{\delta x^k} $ and $ R^i_{jkm}=\frac{\delta\Gamma^{i}_{jm}}{\delta  x^k}-\frac{\delta\Gamma^{i}_{jk}}{\delta x^m}+
\Gamma^{i}_{s k}\Gamma^{s}_{jm}-\Gamma^{i}_{s m}\Gamma^{s}_{jk} $.
For a non-null $ y\in T_xM $,   trace of the  hh-curvature  called the \textit{Riemann curvature} are given by
$ R_y(u)=R^i_ku^k\frac{\partial}{\partial x^i},$     where
 $ R^i_{k}:=g^{mj}R^i_{mjk} $.
 \label{Re;R=R^*}
       Multiplying the components of the hh-curvature tensor of Cartan connection $\overset{\ _*}{R}^i_{jkm}$ in \eqref{R} by $ y^j $ yields
       $ \overset{\ _*}{R}^{i}_{km}=R^i_{km}+0 $. Again contracting this equation by $ y^k $ we have
        \begin{equation}\label{R*}
        \overset{\ _*}{R}^{i}_{m}=R^i_{m}.
         \end{equation}
    The \textit{Ricci scalar} is defined by $ Ric:=R^i_i $,\, see \cite[p.331]{BCS}. Here, we use Akbar-Zadeh's definition of \textit{Ricci tensor} as follows
  $Ric_{ik}:=1/2(F^2Ric)_{y^iy^k} $, see \cite{Ak}.
   Let $l^i:=\frac{y^i}{F}$, be a unitary 0-homogenous vector field,  by homogeneity we have $Ric_{ij} l^il^j=Ric$.
    A Finsler structure  $F$ is called \emph{forward (resp. backward) geodesically complete}, if every  geodesic on an open interval $(a, b)$ can be extended to a geodesic on $(a,\infty)$ (resp. $(-\infty, b)$).
     A Finsler structure is called \emph{``complete"} if it is forward and backward complete.

 Let $\tilde{F} $ be another Finsler structure on $M$. If any geodesic of $(M,F)$  coincides with a geodesic of $(M,\tilde{F})$ as set of points and vice versa, then the change  $F\rightarrow \tilde{F}$ of the metric is called \emph{projective} and $F$ is said to be \emph{projective} to $\tilde{F}$.
   In the definition of projective changes we deal with the forward geodesics and  the word ``geodesic"    refers to the forward geodesic.
   A Finsler space $(M,F)$  is projective to another Finsler space $(M,\tilde{F})$, if and only if there exists a  1-homogeneous scalar field $p(x,y)$ satisfying
       $ \tilde{G}^i(x , y)=G^i(x , y)+p(x,y)y^i$.
      The scalar field $p(x,y)$ is called the \emph{projective factor} of the projective change under consideration.

      Let $F$ and $\bar{F}$ be two Finsler structures on an n-dimensional manifold $M$. A diffeomorphism $f:(M,F)\to(M,\bar{F})$ is called \textit{conformal transformation}  or  simply a \textit{conformal change} of metric,
       if and only if there is a scalar function $\varphi(x)$ on $M$ called the \emph{conformal factor} such as $\bar{F}(x,y)=e^{\varphi(x)}F(x,y)$.
Assuming $\bar{F}(x,y)=e^{\varphi(x)}F(x,y)$ we have equivalently
$
\bar{g}_{ij}(x,y)=e^{2\varphi(x)}g_{ij}(x,y)$ and $ \bar{g}^{ij}(x,y)=e^{-2\varphi(x)}g^{ij}(x,y)
$
where $ g^{ij} $ is the inverse matrix defined by $ g_{ij}g^{ik}=\delta^k_j $.
The diffeomorphism $f$ is said to be \textit{homothetic} if $\varphi$ is constant and \textit{isometric} if $ \varphi $ vanishes in every point of $M$.
 A conformal transformation is called \textit{C-conformal} if the following condition holds
     $
      \varphi_hC^h_{ij}=0
     $,
     where $ \varphi_h=\frac{\partial \varphi}{\partial x^h} $.\\
 Throughout this article, the objects of $({M}, \bar{F})$ are decorated with a bar and we shall always assume that the line elements $(x,y)$ and $(\bar{x},\bar{y})$ on $(M,F)$ and $({M},\bar{F})$ are chosen such that $ \bar{x^i}=x^i$ and $ \bar{y^i}=y^i $ holds, unless a contrary assumption is explicitly made.

Let $(x,y)$ be an element of $ TM$ and  $P(y,X)\subset T_xM $ a 2-plane generated by the vectors $y$ and $X$ in $T_xM $. The  \textit{flag curvature} $ \kappa(x,y,X) $ with respect to the plane $ P(y,X) $ at a point $ x\in M $ is defined by \[\kappa(x,y,X):=\frac{g_(R(X,y)y,X)}{g(X,X)g(y,y)-g(X,y)^2},\] where $ R(X,y)y $ is the hh-curvature tensor. If $ \kappa $ is independent of $ X $, then $ (M,F) $ is called  \textit{ of scalar curvature space}.
If $ \kappa $ has no dependence on $ x $ or $ y $, then it is called  of \textit{constant curvature}, cf. \cite{BCS}. A Finsler structure $F$  on the smooth $n$-dimensional manifold $M$ is called a \textit{Randers metric} if  $ F=\alpha+\beta $, where $\alpha(x,y):=\sqrt{a_{ij}y^iy^j},$  is a Riemannian metric and $\beta(x,y):=b_i(x)y^i,$ is a 1-form. For a detailed  study of Randers metric on Finsler geometry, one can refer to the book \cite{ChSh}.\\
Fix a tangent vector field $T\in T_pM$ and consider the constant speed geodesic $\sigma(t)=exp_p(tT)$, $0\leqslant t\leqslant r$ that emanates from $p=\sigma(0)$ and terminates at $q=\sigma(r)$.
 If there is no risk of confusion, also we indicate its  speed field by $T$.
  Let $D_T$ denote the covariant differentiation along $\sigma$, with reference vector $T$. Recall that a vector field $J$ along $ \sigma $ is said to be a \textit{Jacobi field} if it satisfies the equation $D_TD_TJ+R(J,T)T=0.$
 We say that the point $q\in M$ is \textit{conjugate} to the point $p\in M$ along the geodesic $\sigma$ if there exists a nonzero Jacobi field $J$ along $\sigma$ which vanishes at the both points $p$ and $q$. See \cite[p.173,174]{BCS}.
 A  \textit{fundamental group} at a fixed point $x\in M$   is a natural and informative group consisting of the equivalence classes of homotopy of the loops on the underlying topological space. It contains basic information about the  topology of the  space, like shape, or number of holes. We denote the fundamental group  at a fixed point $x\in M$ by $\pi(M,x)$.
\subsection{ Schwarzian tensor and M\"{o}bius mapping}
Let $h:M\to\mathbb{R}$ be a smooth real function on an $n$-dimensional $(n\geqslant2)$ Finsler manifold $(M,F)$. At a point $p$, we indicate the vector field \textit{gradient} of $h$  by $\nabla h(p)=grad\, h(p)\in \pi^*TM$  which is defined for all $  v\in T_pM,$ by
 $g_{_{grad\,h(p)}}(v, grad\, h(p))=dh_p(v),$ where $dh:=\frac{\partial h}{\partial x^i}dx^i$ is the differential of $h$.
In terms of a local coordinate system, we have $grad\, h:=h^i(x)\frac{\partial}{\partial x^i}\in \pi^*TM$, where $h^i(x)=g^{ij}(x, grad\, h(x))\frac{\p h}{\partial x^j}$.

 For each vector field $Y=Y^i\frac{\partial}{\partial x^i} \in \pi^*T M$, the \emph{horizontal divergence} and the \textit{vertical divergence} of $Y$ are scalar functions on $M$ defined by the contraction of their covariant derivatives, $div^h Y: =\frac{\partial Y^i}{\partial x^i}+\Gamma^i_{ij}Y^j$ and $ div^vY:=C^i_{ij}Y^j $ respectively, in a local coordinate system.


For a real smooth function $h$ on $M$, we have defined the \emph{Hessian} of $h$
 in the Cartan case, as follows, see \cite[p.883]{B.S}
 \begin{align*}
  Hess(h)(\hat{X},Y)=(\hat{X}Y)h-({}^c\nabla_{\hat{X}}Y)h.
 \end{align*}
In a local coordinate system it is written
 \begin{align*}
\big(Hess(h)\big)_{ij}=\frac{\p^2 h}{\p x^i\p x^j}-(\Gamma^k_{ij}+C^k_{ij})\frac{\p h}{\p x^k},
 \end{align*}
where, the Christoffel symbols   $\Gamma^h_{ij}$ and the Cartan torsion $C^h_{ij}$ are given by \eqref{Eq;Chris.Sym.}. As usual, in Finsler space, the \emph{Laplacian} for a real function $h$ on $M$ is defined by the trace of Hessian
 $\Delta h=g^{ij}\{\frac{\p^2 h}{\p x^i\p x^j}-(\Gamma^k_{ij}+C^k_{ij})\frac{\p h}{\p x^k}\}$.


Let $f:(M,F)\to(M,\bar{F})$ be a conformal transformation such that $\bar{F}(x,y)=e^{\varphi(x)}F(x,y)$ where, $\varphi:M\to\mathbb{R}$ is a smooth function on $M$.
  The \emph{Schwarzian derivative} of a conformal map $f:(M,F)\to (M,\bar{F})$ with $\bar{F}=e^{ \varphi}F$, at a point $ x\in M $, is a linear map
  \begin{align*}
   S_F(f):\Gamma(TM_0)  &\longrightarrow \Gamma(\pi^*TM) ,\\
  S_F(f) \hat{ X}&= {}^c\nabla_{\hat{X}}(\nabla \varphi)-g(\nabla\varphi,\varrho\hat{X})\nabla\varphi-\frac{1}{n}(\Delta\varphi-\|\nabla\varphi\|^2) \varrho \hat{X},
  \end{align*}
  where  $ \hat{X}\in \Gamma(TM_0)$ and $\varrho \hat{X}=X$. For more details see \cite{B.S}.

We say that the equation $S_F(f)\hat{ X}=0$ or equivalently
  \bes\label{Eq;Schw.Der}
  {}^c\nabla_{\hat{X}}(Y)-g(Y,\varrho\hat{X})Y-\frac{1}{n}(div Y- \|Y\|^2) \varrho \hat{X}=0,
  \ees
 is \textit{completely integrable} at $ x\in M $ if for every $ Y\in T_xM $, there is a local solution $\varphi(x)$ where $grad \ \varphi(x)=Y$.

   The \emph{Schwarzian tensor} $B_{_F}(\varphi)$ of a smooth function  $\varphi:M\to\mathbb{R}$ on $(M,F)$ is a  symmetric traceless $(0,2)$-tensor field  defined by
 \begin{align}\label{B-phi}
 B_{_F}(\varphi)(\hat{X},Y)=\textrm{Hess}(\varphi)(\hat{X},Y)-(d\varphi\otimes d\varphi)(\varrho\hat{X},Y)-\frac{1}{n}(\Delta\varphi-\|grad\varphi\|^2)g(\varrho\hat{X},Y),
 \end{align}
 for all $\hat{X}\in \Gamma(TM_0)$ and $Y\in \Gamma(\pi^*TM)$  where   $\|grad\varphi\|^2=\varphi^i\varphi_i$, \, $\varphi^i=g^{ij}\varphi_j$  and
 $g$ is the inner product on $\pi^*TM$ derived from the Finsler structure $F$, see \cite[p.885]{B.S}.\\

   A conformal diffeomorphism $f:(M,F)\to (M,\bar{F}),$ is called a \emph{M\"{o}bius mapping}, if the Schwarzian derivative $S_{_F}(f)$ vanishes.  By means of $ S_F(f)=B_F{(\varphi)} $, one can show that  $f$ is a M\"{o}bius mapping, if and only if the Schwarzian tensor   $B_F(\varphi)$, vanishes. \\
   In terms of a local coordinate system, \eqref{B-phi} becomes
     \begin{equation}\label{a8}
     \big(B_F(\varphi)\big)_{ij}={}^c\nabla_i\varphi_j-\varphi_i\varphi_j
     -\frac{1}{n}(\Delta\varphi-\|grad\varphi\|^2)g_{ij},
      \end{equation}
    where, ${}^c\nabla_i\varphi_j:=\frac{\partial^2\varphi}{\partial x^i\partial x^j} - ( \Gamma^h_{ij}+C^h_{ij})\varphi_h$ are the components of Cartan h-covariant derivative.


   \section{Finsler manifolds of constant curvature}
      Here, we find an integrability condition for the system of  
      M\"{o}bius  partial differential equations $ (B(\varphi))_{ij}=0 $.\\
  \emph{\textbf{Proof of Theorem \ref{Th;Z=0}}}.
    Let $(M,F)$ be a Finsler manifold admitting a non-trivial M\"{o}bius mapping. By definition, the Schwarzian derivative vanishes and  \eqref{a8} leads to
    \bes\label{Eq;secondOrder}
    {}^c\nabla_j\varphi_i - \varphi_i \varphi_j=\Phi g_{ij},
    \ees
    where we set  $\Phi=\frac{1}{n} (\Delta\varphi - ||grad\varphi||^2)$.
 Applying the Cartan horizontal derivative to the both sides of the last equation and replacing again ${}^c \nabla_j\varphi_i $ yields
     \begin{equation}\label{a11}
   {}^c\nabla_k{}^c\nabla_j\varphi_i=2\varphi_i\varphi_j\varphi_k+\Phi(g_{ik}\varphi_j+g_{jk}\varphi_i)+\Phi_kg_{ij},
   \end{equation}
    where, $ \Phi_k:={}^c\nabla_k\Phi=\delta\Phi/\delta x^k $.
    Consider the following well-known Ricci identity: cf. \cite[p.121]{An},
           \begin{equation}\label{Ric identity}
            {}^c\nabla_k{}^c\nabla_j\varphi_i-{}^c\nabla_j{}^c\nabla_k\varphi_i=
            -\overset{\ _*}{R}^h_{ijk}\varphi_h-R^h_{jk}\varphi_{h;i},
           \end{equation}
            where, $ \varphi_{h;i}=\dot{\partial}_i\varphi_h-C^s_{hi}\varphi_s=-C^s_{hi}\varphi_s $.
     Replacing \eqref{a11} in the Ricci identity  we get
   \begin{equation}\label{a2}
    g_{ij}(\Phi_k-\varphi_k\Phi)-g_{ik}(\Phi_j-\varphi_j\Phi)=
    -\overset{\ _*}R^h_{ijk}\varphi_h-R^h_{jk}\varphi_{h;i}.
    \end{equation}
     Substituting $ \varphi_{h;i} $ and \eqref{R}   in \eqref{a2} yields
 \begin{equation}\label{a3}
   g_{ij}(\Phi_k-\varphi_k\Phi)-g_{ik}(\Phi_j-\varphi_j\Phi)=-R^h_{ijk}\varphi_h.
    \end{equation}
   Contracting  by $ g^{ij} $ we obtain $ (n-1)(\Phi_k-\varphi_k\Phi)=-R^h_{\ k}\varphi_h $. Replacing in \eqref{a3} yields
  \begin{equation}\label{a4}
    \frac{1}{n-1}(g_{ij}R^h_{k}-g_{ik}R^h_{j})\varphi_h=R^h_{ijk}\varphi_h.
    \end{equation}
  Using the last equation, we consider a tensor field $Z$ with the components $ Z^h_{ijk} $ defined by the equation \eqref{tensor Z}
   where,  $R^h_{ijk}$ is the $hh$-curvature of Chern connection. Equation \eqref{a4} yields $ Z^h_{ijk}\varphi_h=0$.
The M\"{o}bius partial differential equations  ${}^c\nabla_j\varphi_i = \varphi_i \varphi_j+\Phi g_{ij}$ is completely integrable, if and only if
       \bes
       {}^c\nabla_k{}^c\nabla_j\varphi_i-{}^c\nabla_j{}^c\nabla_k\varphi_i=0.
        \ees
        Replacing \eqref{R} in the Ricci identity \eqref{Ric identity},  we obtain
\be\label{nabla,R}
{}^c\nabla_k{}^c\nabla_j\varphi_i-{}^c\nabla_j{}^c\nabla_k\varphi_i=
            -{R}^h_{ijk}\varphi_h.
\ee
From the last two equations, we get $ {R}^h_{ijk}\varphi_h=0$. By definition of complete integrability, this relation holds for any initial data $\varphi_h(x)$, hence ${R}^h_{ijk}=0$.
 Therefore  $ R^h_j=0 $, and  $ Z^h_{ijk}=0 $.

 Conversely, let $ Z=0 $, by definition  we have
       \be\label{R2}
       R^h_{ijk}=\frac{1}{n-1}(g_{ij}R^h_{k}-g_{ik}R^h_{j}).
       \ee
Contracting the both sides of \eqref{R2} with $y^i$   and  using  $R^{i}_{km}:=y^j R^{i}_{ jkm}$ yields
\bes
R^h_{jk}=\frac{1}{n-1}(y_{j}R^h_{k}-y_{k}R^h_{j}).
\ees
 Contracting again   with $y^k$ and making use of  $ R^h_ky^k=0 $,  cf. \cite[p.55,57]{BCS}, we obtain
 \bes
  R^h_{j}=\frac{-F^2 R^h_{j}}{n-1},
 \ees
 hence  $ R^h_j=0 $. By means of \eqref{R2}, we get $ R^h_{ijk}=0 $ and  \eqref{nabla,R} leads
 $ {}^c\nabla_k{}^c\nabla_j\varphi_i-{}^c\nabla_j{}^c\nabla_k\varphi_i=0 $. Therefore the  M\"{o}bius partial differential equations $ (B(\varphi))_{ij}=0 $, is completely integrable and proof of Theorem \ref{Th;Z=0} is complete.
   \hspace {\stretch{1}}$\Box$

 \bco\label{Th;constant curv}
 Let  $(M,F)$  be a Finsler manifold. If it is of constant curvature, then the $Z$ tensor vanishes.
 \eco
 \bpf   Let $(M,F)$ be a Finsler manifold, the components of Cartan $hh$-curvature tensor $\overset{\ _*}{R}_{ijkl}$ are given by
  \[\overset{\ _*}{R}_{ijkl}=\kappa(g_{ik}g_{jl}-g_{il}g_{jk})+\kappa F^2Q_{ijkl}+1/2\nabla_0\nabla_0Q_{ijkl}, \]
   where, $Q_{ijkl}$ are the components of $vv$-curvature of Cartan connection.
   If  the flag curvature $\kappa$ is constant, then one can see that $Q_{ijkl}=0$,  cf. \cite[p.26]{Ak} and we have
  \be\label{Eq;Constflag+1}
     \overset{\ _*}{R}^{h}_{ijk}=\kappa(g_{ij}\delta^h_k-g_{ik}\delta^h_j).
   \ee
          Contracting the both sides of the last equation with $g^{ij}$, and using \eqref{R*} yields
           \be\label{Eq;Constflag+2}
          \overset{\ _*}{R}^h_k =R^h_k=\kappa(n-1)\delta^h_k.
           \ee
   The equation \eqref{Eq;Constflag+1} holds well for the  $hh$-curvature tensor of Chern(Rund) connection, \cite[p.109]{An}.
   Replacing the last two equations \eqref{Eq;Constflag+1} and \eqref{Eq;Constflag+2} in \eqref{tensor Z} yields
\bes
 Z^h_{ijk}=\kappa(g_{ij}\delta^h_k - g_{ik}\delta^h_j) - \kappa(g_{ij}\delta^h_k - g_{ik}\delta^h_j)=0,
 \ees
 and we have proof of this corollary.
 \epf
\section{Applications of integrability condition  of Schwarzian derivative}
By studying the integrability condition, the following results are obtained.\\

       \emph{\textbf{Proof of Theorem \ref{Th;Johari}.}}~
        Let $(M,F)$ be a connected complete Finsler $n$-manifold of constant Ricci curvature $c^2$, admitting a non-trivial M\"{o}bius mapping.
         The Schwarzian integrability condition $S_F(f)\hat{X}=$0 and  \eqref{a8} lead  to
    \be\label{a9+}
     {}^c\nabla_i\varphi_j - \varphi_i \varphi_j=\Phi g_{ij},
     \ee
    where we put  $\Phi=1/n (\Delta\varphi - ||grad\varphi||^2)$.
    Therefore  $(M,F)$ admits a non-trivial M\"{o}bius mapping which is a non-trivial conformal change of metric  $\bar{g}=e^{2\varphi}g $, satisfying the M\"{o}bius equation \eqref{a9+}. After changing the variable $\rho=e^{-\varphi}$, in the equation \eqref{a9+},  a simple calculation yields   $\varphi_l=-\rho_l/\rho$, where $\rho $ is a positive real function on $M$ and  $ \rho_l=\partial \rho/\partial x^l$. Hence, we have \[{}^c\nabla_k\varphi_l=-\frac{\rho\,{^c}\nabla_k\rho_l-\rho_k\rho_l}{\rho^2}. \]
   Replacing these two terms in  the M\"{o}bius equation \eqref{a9+}
yields,
\begin{equation}\label{nabla rho}
 {}^c\nabla_k\rho_l=\phi g_{lk},
 \end{equation}
 where $\phi=-\rho\Phi$. Therefore, vanishing of the M\"{o}bius equation $B_F(\varphi)=0$, is equivalent to the equations  \eqref{a9+} and \eqref{nabla rho}.
Replacing $ \phi=-c^2\rho $,  the  equation \eqref{nabla rho} becomes ${}^c\nabla_j\rho_k+c^2\rho g_{jk}=0$.
    The following theorem in \cite{JoBi}  completes the proof of Theorem \ref{Th;Johari}. \hspace{\stretch{1}}$\Box$
     \begin{theo}\cite{JoBi}\label{the;Johari}
     Let $ (M,g) $ be a complete connected Finsler manifold of constant Ricci curvature $ c^2 $. If $ M $ admits a non constant function $ \rho $ satisfying the ODE;  $${}^c\nabla_i{}^c\nabla_j\rho+c^2\rho g_{ij}=0,$$
      then $ M $ is homeomorphic to the n-sphere.
     \end{theo}

     \subsection{An integrability condition for Finsler manifolds of scalar curvature}
           Here we study an integrability condition for  the Schwarzian tensor $ B_F(\varphi) $, related to the Finsler manifolds of scalar curvature.
     \begin{thm}\label{thm;z=0}
     The M\"{o}bius partial differential equation $B_F(\varphi)=0 $ is completely integrable if and only if
     $Z^h_{jk}=0$, where
    \begin{equation}\label{Z scalar}
     Z^h_{jk}=R^h_{jk}-F^{-2}(y_jR^h_k-y_kR^h_j).
     \end{equation}
    \end{thm}
  \begin{proof}
  Let us consider the partial differential $ B_{ij}(\varphi)=0 $, which is equivalent to the equation \eqref{a9+}. The horizontal Cartan covariant derivative of $ {}^c\nabla_i\varphi_j - \varphi_i \varphi_j=\Phi g_{ij} $ and the Ricci identity \eqref{Ric identity}, together with a similar procedure as in the proof of Theorem \ref{Th;Z=0}, yields \eqref{a3}. Contracting \eqref{a3} by $ y^i $ gives
  \begin{equation}\label{b12}
              y_j(\Phi_k-\varphi_k\Phi)-y_k(\Phi_j-\varphi_j\Phi)=-R^h_{jk}\varphi_h,
    \end{equation}
   where, $ y_i=g_{ij}y^j $. Again contracting \eqref{b12} by $ y^j $ we have
  \begin{equation}\label{b2}
  (\Phi_k-\varphi_k\Phi)=F^{-2}(y_k(\Phi_0-\varphi_0\Phi)-R^h_{\ k}\varphi_h),
    \end{equation}
   where $ \Phi_0=\Phi_jy^j $, $ \varphi_0=\varphi_jy^j $ and $ R^i_k=R^i_{0k} $.
    Substituting \eqref{b2} in \eqref{b12} we get
     \begin{equation*}
    (R^h_{jk}-F^{-2}(y_jR^h_k-y_kR^h_j) )\varphi_h=0.
    \end{equation*}
     By means of  \eqref{Z scalar}, it yields $ Z^h_{jk}\varphi_h=0 $. If the partial differential $ B_{ij}(\varphi)=0 $, is completely integrable then the relation $ Z^h_{jk}\varphi_h=0 $ satisfies with any initial data $ \varphi_h $, therefore $ Z^h_{jk}=0$.
     Conversely, let $Z^h_{jk}=0$, the equation \eqref{Z scalar} yields
\[
R^h_{jk}=F^{-2}( y_j R^h_k- y_k R^h_j).
\]
Contracting the both sides of the last equation with $y^k$ and using $y^k R^h_k=0$ leads
$R^h_j=F^{-2}(-F^2R^h_j)$, hence $R^h_j=0$.
From \eqref{R2}   we get $R^h_{ijk}=0$ and \eqref{nabla,R} results  the Mobius partial differential equations $B_F(\varphi)=0$ is completely integrable and we have the proof.
\end{proof}
 Now we are in a position to prove the following corollary.
 \bco\label{Th;a,b,c}
 Let $ (M,F) $ be a connected complete Finsler n-manifold of scalar curvature. If $ (M,F) $ admits a  M\"{o}bius mapping, then it is conformal to one of the following spaces;\\
 \textbf{(a)}  A direct product $ I\times N $ of an open interval $ I $ of the real line and an $(n-1)$-dimensional complete Finsler manifold $ N $.\\
 \textbf{(b)} An n-dimensional Euclidean space;\\
 \textbf{(c)} An n-dimensional unit sphere in an Euclidean space.
 \eco
 \begin{proof}
  Let $ (M,F) $ be a Finsler manifold of scalar curvature admitting a non-homothetic conformal change, that is, there is a non-constant scalar function $ \varphi $ on $ M $, satisfying $ \bar{g_{ij}}=e^{\varphi(x)}g_{ij} $. A Finsler manifold is isotropic or of scalar curvature if and only if we have
   \begin{equation}\label{c11}
  \overset{\ _*}R^h_j=\kappa F^2(\delta^h_j-l^hl_j),
   \end{equation}
     where, $\kappa$ is the flag curvature and $l^i=\frac{y^i}{F} $, $ l_i=\frac{y_i}{F}$, or equivalently if and only if
     \begin{equation}\label{c22}
    \overset{\ _*}R^h_{jk}=\kappa F(l_k\delta^h_j-l_j\delta^h_k),
      \end{equation}
        where $  R^h_j=R^h_{jk}y^k $ and $  R^h_{jk}=R^h_{jkm}y^m $, see \cite[p.133-147]{Ru}.
     Recall that  \eqref{R*} claims $ \overset{\ _*}R^h_j= R^h_j $ and $ \overset{\ _*}R^h_{jk}= R^h_{jk} $. In fact by means of \eqref{c22} we have
        \[ R^h_j=R^h_{jk}y^k=y^k\kappa F(l_k\delta^h_j-l_j\delta^h_k)=\kappa F(F\delta^h_j-y^hl_j)=\kappa F(F\delta^h_j-l_jl^hF )=\kappa F^2(\delta^h_j-l_jl^h ).   \]
      Replacing \eqref{c11} and \eqref{c22} in \eqref{Z scalar} we obtain
    \begin{equation*}
    Z^h_{jk}=\kappa F(l_k\delta^h_j-l_j\delta^h_k)-F^{-2}(y_j(\kappa F^2(\delta^h_k-l^hl_k))-y_k(\kappa F^2(\delta^h_j-l^hl_j))=0.
    \end{equation*}
 We have $Z^h_{jk}=0$,  and by means of Theorem \ref{thm;z=0}, the partial differential equation \eqref{a9+} has non-trivial solutions. Next, from \eqref{nabla rho} consider the  M\"{o}bius equation $ {}^c\nabla_k\rho_l=\phi g_{lk}$,
  where $ \phi=-\rho\Phi $. Thus, if there exists a conformal change on $ (M,F) $, then there is a non-trivial solution $ \rho $ on $ M $ for \eqref{nabla rho}. Now if we assume  $ (M,F) $ is connected and complete, then as a consequence of Theorem \ref{Th;p1} we have proof of the corollary.
   \end{proof}
 \begin{theo}\cite{AsBi1}\label{Th;p1}
   Let $ (M,F) $ be a connected complete Finsler manifold of dimension $ n\geqslant 2 $. If $ M $ admits a non-trivial solution of
    \[ {}^c \nabla_j\rho_k=\phi g_{jk},\]
     where  $ \phi $ is certain function on $ M $, then depending on the number of critical points of $ \rho $, i.e. zero, one or two respectively, it is conformal to\\
     \textbf{(a)} A direct product $ I\times N $ of an open interval $ I $ of the real line and an $(n-1)$-dimensional complete Finsler manifold $ N $.\\
     \textbf{(b)} An n-dimensional Euclidean space;\\
      \textbf{(c)} An n-dimensional unit sphere in an Euclidean space.
      \end{theo}
The following proposition shows the relationship between M\"{o}bius mapping and C-conformal transformations.
     \bpr
     Every M\"{o}bius mapping  is a C-conformal transformation.
     \epr
      \begin{proof}
    Let $f:(M,F)\to (M,\bar{F})$ be a M\"{o}bius mapping. By definition
     $B_{ij}(\varphi)=0$ that is
     \begin{equation}\label{B=0}
     \varphi_{ij}-(\Gamma^h_{ij}+C^h_{ij})\varphi_h-\varphi_i\varphi_j-\Phi g_{ij}=0,
     \end{equation}
     where $ \Phi=\frac{1}{n}(\Delta\varphi-\|\nabla\varphi\|^2) $, $ \varphi_i=\frac{\partial \varphi}{\partial x^i} $  and $ \varphi_{ij}=\frac{\partial^2\varphi}{\partial x^i\partial x^j} $.
     By differentiating  \eqref{B=0} by $y^k$ we have
     \begin{equation*}
     -\Gamma^h_{ijk}\varphi_h-C^h_{ijk}\varphi_h-\Dot{\Phi}_kg_{ij}-\Phi(2C_{ijk})=0,
     \end{equation*}
      where $\dot{\Phi}=\frac{\partial \Phi}{\partial y^k} $,  $\Gamma^h_{ijk}:=\frac{\p }{\p y^k}\Gamma^h_{ij}$ and
     $C^h_{ijk}= \frac{\partial}{\partial y^k}C^h_{ij}$.
     Contracting the both sides of the last equation with $ y^k $ yields
     \begin{equation}\label{B2}
     -\Gamma^h_{ijk}\varphi_hy^k-C^h_{ijk}\varphi_hy^k-\dot{\Phi}_{k}g_{ij}y^k=0.
     \end{equation}
A moment's thought shows that the components of the Christoffel symbols $\Gamma^h_{ij}$ given in \eqref{Eq;Chris.Sym.}
 are positively homogeneous of degree $(0)$ since all of its three terms  $\delta_k g_{jh}=\frac{\partial g_{jh}}{ \partial x^k}- G^i_k \frac{\partial g_{jh}}{\partial y^i}$  are of degree $(0)$. In fact, $g_{jh}$ and its derivative with respect to $x^k$   are 0-homogeneous and $G^i_k$ are homogeneous of degree $(1)$.
 Therefore, $\Gamma^h_{ijk}=\frac{\p }{\p y^k}\Gamma^h_{ij}$  are homogeneous of degree $(-1)$.
As well, the components of Cartan tensor $C_{ijk} $ are positively homogeneous of degree $(-1)$ and $ C^h_{ijk} $ is positively homogeneous of degree $(-2)$. Therefore  Euler's theorem implies $\Gamma^h_{ijk}y^k=0=C_{ijk}y^k$
 and $ -C^h_{ijk}y^k=C^h_{ij} $, hence \eqref{B2} yields
     \begin{equation}\label{C,Phi}
      C^h_{ij}\varphi_h-\dot{\Phi}_kg_{ij}y^k=0.
     \end{equation}
 Contracting  the both sides of  the above equation with $y^iy^j$ and using $ g_{ij}y^iy^j=F^2 $, we obtain
$F^2\dot{\Phi}_ky^k=0$, hence $ \dot{\Phi}_ky^k=0 $.
    Therefore \eqref{C,Phi} yields
    $ C^h_{ij}\varphi_h=0 $ which completes the proof.
     \end{proof}
   \section{Schwarzian and Bonnet-Myers theorem}
   The following  Bonnet-Myers type theorem confirms Thurston's hypothesis and shows that the Schwarzian derivative of a projective parameter plays an identical role to the Ricci curvature on complete Finsler manifolds.
   \begin{theo}\cite[p.194]{BCS}\label{bonnet}
   Let $(M,F)$ be an $n$-dimensional forward-complete connected Finsler manifold. Suppose its Ricci curvature has the uniform positive lower bound
   \[Ric\geqslant(n-1)\lambda>0;\]
   equivalently, $ y^iy^jRic_{ij}(x,y)\geqslant(n-1)\lambda F^2(x,y) $, with $ \lambda>0 $. Then:\\
   $(i)$ Every geodesic of length at least $ \pi/ \sqrt{\lambda} $ contains conjugate points.\\
   $ (ii) $ The diameter of $ M $ is at most $ \pi/ \sqrt{\lambda} $.\\
   $ (iii) $ $ M $ is in fact compact.\\
   $ (iv) $ The fundamental group $ \pi(M,x) $ is finite.
   \end{theo}
Using the approximation of the Schwarzian derivative, we can characterize the forward-complete Finsler manifolds.\\

   \emph{\textbf{Proof of Theorem \ref{Th;Bonnet-Myers}.}}
   Let $(M,F)$ be an $n$-dimensional  connected Finsler manifold and $\gamma(t)$ a geodesic on $(M,F)$.
 In general, the parameter $t$ in  $\gamma(t)$ does not remain invariant under the projective changes of $F$. There is a unique parameter up to linear fractional transformation which remains invariant under a projective change  of Finsler structure, called projective parameter, see \cite{SBi}.
In fact let $p(s)$ be a projective parameter on $(M,F)$, where $s$ is the arc length parameter of the geodesic $\gamma$.  Schwarzian derivative of the projective parameter $p(s)$ is given by;
\bes
S(p(s))=\frac{\frac{d^3p}{ds^3}}{\frac{dp}{ds}}-\frac{3}{2}\Big[\frac{\frac{d^2p}{ds^2}}{\frac{dp}{ds}}\Big]^2, \ees
and the projective parameter $p$ is a solution of the above ODE.
One can show that, the projective parameter $p(s)$ is unique up to a linear fractional transformations, that is
       \begin{equation*}
       S(p\circ T)=S(p),
       \end{equation*}
      where $ T=\frac{ax+b}{cx+d}$ and $ad-bc\neq 0$.
It is well known that,
      \begin{equation} \label{sch,Ric}
      S\big(p(s)\big) =\frac{2}{n-1}Ric_{jk}\frac{d{x}^j}{ds}\frac{d{x}^k}{ds}=\frac{2}{n-1}F^2Ric,
      \end{equation}
           where $ S\big(p(s)\big) $ is the Schwarzian of \textquotedblleft $ p(s) $\textquotedblright and \textquotedblleft $ s $\textquotedblright is the arc length parameter, see
           \cite{BiS}.

       When  the Ricci tensor is parallel with respect to any of Berwald, Chern or Cartan
      connection, then it is constant along the geodesics and we can easily solve the equation (\ref{sch,Ric}),   see for more details \cite[page 5]{BiS}.

    Let the Schwarzian of the projective parameter  $p$ satisfies, $S(p)\geqslant2F^2\lambda$,  where $ \lambda>0 $, is a positive number. Equation \eqref{sch,Ric} yields
   \[\frac{2}{n-1}F^2Ric\geqslant2F^2\lambda,\]
   hence $ Ric\geqslant(n-1)\lambda>0.$ Assuming $(M,F)$ is forward-complete, the proof is a consequence of  Bonnet-Myers theorem \ref{bonnet}.
   \hspace {\stretch{1}}$\Box$\\
      In order to approximate other smooth functions on a Finsler manifold with compact support one can use the method explained in \cite{BiSh}.

\textbf{ Acknowledgement.}
The first author would like to thank the ``Institut de Math\'{e}matiques de Toulouse" (ITM) at the Paul Sabatier University of Toulouse, where this article is partially written.

       Behroz Bidabad\\
    Department of Mathematics and Computer Sciences\\
    Amirkabir University of Technology (Tehran Polytechnic),
    424 Hafez Ave. 15914 Tehran, Iran.
    E-mail: bidabad@aut.ac.ir\\
    Faranak Seddighi\\
    Faculty of Mathematics, Payame Noor University of Tehran, Tehran, Iran.\\
    f-seddighi@student.pnu.ac.ir


\begin{thebibliography}{0}
    \bibitem{Ak} H. Akbar-Zadeh, Sur les espaces de Finsler \'{a} courbures sectionelles constantes, Acad. Roy. Belg. Bull. Cl. Sci. (5) 74 (1988), 281-322.
    \bibitem{An} P. L. Antonelli, R. S. Ingarden and M. Matsumoto, The Theory of sprays and Finsler Spaces with Applications in Physics and Biology, FTPH 58, Kluwer Academic Publishers, 1993.
    \bibitem{AsBi1} A. Asanjarani and B. Bidabad, A classification of complete  Finsler manifolds through a second order diff. equation, Differential Geometry and its Applications 26(2008), 434-444.
    \bibitem{BCS} D. Bao, S. Chern and Z. Shen, An introduction to Riemann-Finsler Geometry, Springer-Verlag, 2000.
    \bibitem{B.S} B. Bidabad and F. Sedighi, The Schwarzian derivative and conformal transformation on Finsler manifolds, J. Korean Math. Soc., 57, No. 4, (2020) 873-- 892.

    \bibitem{BiS}   B.  Bidabad and M. Sepasi, \emph{On a projectively invariant pseudo-distance in Finsler Geometry,} Int. J. Geom. Methods Mod. Phys. 12(4) (2015), 1550043.
    \bibitem{BiS2}    B.  Bidabad and M. Sepasi, \emph{On complete Finsler spaces of constant negative Ricci curvature,} Int. J. Geom. Methods Mod. Phys.  17,(3) (2020), 2050041.
    \bibitem{BiSh}
    B.  Bidabad and A. Shahi, \emph{On Sobolev spaces and density theorems on Finsler manifolds}, AUT J. Math.  Com., 1(1) (2020)  37-45.
    \bibitem{Ca} K. Carne, \emph{The Schwarzian derivative for conformal maps}, J. Reine Angew. Math. 408 (1990), 10-33.
    \bibitem{ChSh} X. Cheng and Z. Shen, \emph{Finsler geometry. An approach via Randers spaces}, Science Press Beijing; Springer, Heidelberg, 2012.
    \bibitem{JoBi} P. Joharinad and B. Bidabad,\emph{ Conformal vector fields on complete Finsler spaces of constant Ricci curvature}, Differential Geometry and its Applications, 33 (2014), 75-84.
     \bibitem{OS}  B. Osgood and D. Stowe, \emph{The Schwarzian derivative and conformal mapping of Riemannian manifolds}, Duke Math. J. 67 (1992), 57-99.
     \bibitem{Ru} H. Rund, \emph{The Differential Geometry of Finsler spaces,} Springer-Verlag, 1959.
     \bibitem{SBi} M. Sepasi, B. Bidabad, \emph{ On a projectively invariant distance on Finsler spaces,} C. R. Acad. Sci. Paris, Ser. I 352, (2014), 999--1003.
     \bibitem{Sh2} Z. Shen, \emph{On projectively related Einstein metrics in Riemann-Finsler geometry,}  Math. Ann. 320, 625--647 (2001).
          \bibitem{ShYa} Z. Shen, and G. Yang, \emph{On concircular transformations in Finsler geometry,} Results Math. 74, 4 (2019), 25 pp.
      \bibitem{Th}  W.P.  Thurston, \emph{Zippers and univalent functions, The Bieberbach conjecture} (West Lafayette, Ind., 1985),  Math. Surveys Monogr., 21, Amer. Math. Soc., Providence, RI, (1986), 185--197.
    \end{thebibliography}
    \end{document}